\documentclass[11pt]{amsart}
\usepackage[arrow,matrix,curve,color]{xy}
\usepackage{pb-diagram,pb-xy}
\usepackage{graphicx}
\usepackage{amsfonts}
\usepackage{amsthm}
\usepackage{amsmath}
\usepackage{amssymb}
\usepackage{hyperref}   
\usepackage{enumitem}
\usepackage{tikz}

\theoremstyle{plain}
\newtheorem{theorem}{Theorem}[section]
\newtheorem{corollary}[theorem]{Corollary}
\newtheorem{lemma}[theorem]{Lemma}
\newtheorem{proposition}[theorem]{Proposition}
\theoremstyle{definition}
\newtheorem{remark}[theorem]{Remark}
\newtheorem{definition}[theorem]{Definition}
\newtheorem{example}[theorem]{Example}

\newtheorem{question}[theorem]{Question}

\newcommand{\CP}{\mathbb{CP}}

\newcommand{\p}{\partial}
\newcommand{\fb}{\mbox{-}\mathrm{fb}}
\newcommand{\spin}{\mathrm{spin}}
\newcommand{\spinc}{\mathrm{spin}^c}

\newcommand{\MSpin}{\mathsf{MSpin}}

\newcommand{\MPin}{\mathsf{MPin}}

\newcommand{\ko}{\mathsf{ko}}

\newcommand{\psc}{positive scalar curvature}
\newcommand{\co}{\colon\,}

\newcommand{\bR}{\mathbb R}
\newcommand{\bC}{\mathbb C}

\newcommand{\bF}{\mathbb F}
\newcommand{\bH}{\mathbb H}

\newcommand{\bZ}{\mathbb Z}

\newcommand{\bP}{\mathbb P}

\newcommand{\cL}{\mathcal L}

\newcommand{\cO}{\mathcal O}

\newcommand{\cR}{\mathcal R}

\newcommand{\SO}{\mathop{\rm SO}}
\newcommand{\SU}{\mathop{\rm SU}}
\newcommand{\U}{\mathop{\rm U}}

\newcommand{\Sp}{\mathop{\rm Sp}}
\newcommand{\Spin}{\mathop{\rm Spin}}
\newcommand{\Pin}{\mathop{\rm Pin}}

\newcommand{\PSp}{\mathop{\rm PSp}}

\newcommand{\tM}{\widetilde M}

\newcommand{\lp}{\textup{(}}
\newcommand{\rp}{\textup{)}}

\newcommand{\ind}{\operatorname{ind}}

\newcommand{\image}{\operatorname{image}}

\newcommand{\Dirac}{\partial\!\!\!/}
\newcommand{\Cl}{\mathcal{C}\ell}

\newcommand{\bord}{\rightsquigarrow} 

\begin{document}
\title[Positive scalar curvature on $\mathbf{Pin}^\pm$- and
  $\mathbf{Spin}^c$-manifolds]{Positive
  scalar curvature\\ on $\mathbf{Pin}^\pm$- and
  $\mathbf{Spin}^c$-manifolds \\ and manifolds
  with singularities} \author{Boris Botvinnik} \address{Department of
  Mathematics\\ University of Oregon\\ Eugene OR 97403-1222, USA}
\email[Boris Botvinnik]{botvinn@uoregon.edu, bbotvinn@gmail.com}
\urladdr{http://pages.uoregon.edu/botvinn/} \author{Jonathan
  Rosenberg} \address{Department of Mathematics\\ University of
  Maryland\\ College Park, MD 20742-4015, USA} \email[Jonathan
  Rosenberg]{jmr@math.umd.edu}
\urladdr{http://www2.math.umd.edu/\raisebox{-3pt}{~}jmr/}
\begin{abstract}
It is well-known that spin structures and Dirac operators
play a crucial role in the study of positive scalar curvature metrics
(psc-metrics) on compact manifolds. Here we consider a class of non-spin
manifolds with ``almost spin'' structure, namely those with $\spin^c$ or
$\mathrm{Pin}^\pm$-structures.  It turns out that in those cases
(under natural assumptions on such a manifold $M$), the index of a
relevant Dirac operator completely controls existence of a psc-metric
which is $S^1$- or $C_2$-invariant near a ``special submanifold'' $B$
of $M$.  This submanifold $B\subset M$ is dual to the complex
(respectively, real) line bundle $L$ which determines
the $\spin^c$ or pin$^\pm$ structure on
$M$. We also show that these manifold pairs $(M,B)$ can be interpreted as
``manifolds with fibered singularities'' equipped with ``well-adapted
psc-metrics''. This survey is based on our recent work as well as on
our joint work with Paolo Piazza.
\end{abstract}
\maketitle

\section{Introduction}
\label{sec:intro}

\subsection{Motivation} 
It is well-know that from the view-point of differential geometry,
and especially problems involving scalar curvature,
there is a dramatic difference between spin manifolds and non-spin
manifolds. It is easy to check this condition: a smooth compact
oriented manifold $M$ admits a spin structure if and only if its second
Stiefel-Whitney class $w_2(M)=0$. In this paper we will be interested
in the question of when a manifold or pseudomanifold admits a Riemannian
metric of positive scalar curvature, which we abbreviate for
convenience to ``psc-metric.''
Now we recall:
\begin{theorem}[{\cite{MR577131}}]
\label{thm:GLnonspin}  
Let $M$ be a
compact non-spin simply-connected manifold with $\dim M=n\geq 5$.
Then $M$ admits a psc-metric.
\end{theorem}
The situation for spin manifolds is completely different, and
in the simply connected case can be expressed as follows:
\begin{theorem}[{\cite{MR577131,MR1189863}}]
\label{thm:Stolz}  
Let $M$ be a
compact spin simply-connected manifold with $\dim M=n\geq 5$.
Then $M$ admits a psc-metric if and only if $\alpha(M)=0$ in
$KO_n$.
\end{theorem}
Here $KO_n$ is the $n$-the coefficient group for real $K$-theory,
which by the Bott periodicity theorem is given by $\bZ$ for $n$
divisible by $4$, by $\bZ/2$ if $n\equiv 1\text{ or }2\mod 8$,
and is $0$ otherwise.  The invariant $\alpha(M)$ is Atiyah's
$\alpha$-invariant, which can be identified with the $KO_n$-valued
index of the $\Cl_n$-linear Dirac operator defined by the spin structure,
and the map $\alpha$ passes to a \emph{surjective} homomorphism
$\Omega^\spin_n\to KO_n$.  Thus there are plenty of simply connected
spin manifolds that \emph{do not} admit a psc-metric.

We consider two classes of manifolds which are non-spin, but in some sense
are very close to be spin, namely, $\mathrm{pin}^\pm$-manifolds and
$\spinc$-manifolds which
are not spin. These conditions are easy to verify: if $M$ is
not orientable but its orientable double cover has a spin structure,
then $M$ has a $\mathrm{pin}^\pm$-structure. If $M$ is orientable
and $w_2(M)\neq 0$, then $M$ has a $\spinc$-structure
exactly when there is a class $c\in H^2(M;\bZ)$
which maps to $w_2(M)$ under the mod-$2$ reduction map
\begin{equation*}
H^2(M;\bZ)\to H^2(M;\bZ/2).
\end{equation*}
The class $c$ gives a map $c\co M \to \CP^{\infty}$ and,
consequently, a complex line bundle $L \to M$. We use the notation
$(M,L)$ for a manifold with a choice of $\spinc$-structure.

Let $(M,L)$ be $\spinc$-manifold. Since
$M$ is finite-dimensional, the image of the map
$c\co M \to  \CP^{\infty}$ is contained in $\CP^k\subset \CP^{\infty}$
for some $k$. Suppose the map $c$ is transverse to $\CP^{k-1}$, and
let $B= c^{-1}(\CP^{k-1})$:
\begin{equation*}
\begin{diagram}
\setlength{\dgARROWLENGTH}{1.7em}
\node{M}
\arrow{e,l}{c}
\node{\CP^{k}}
\\
\node{B} 
\arrow{n,r,J}{\mathrm{inclusion}}
\arrow{e,l}{c|_{B}}
\node{\CP^{k-1}} 
\arrow{n,r,J}{\mathrm{inclusion}}
\end{diagram}
\end{equation*}
In particular, $L|_B \to B$ is the normal bundle of the inclusion
$B\hookrightarrow M$. The submanifold $B$ is \emph{dual to $L\to M$}.

We denote by $N(B)$ a tubular neighborhood of $B\hookrightarrow M$. Then
$M$ is decomposed as
\begin{equation*}
M = X\cup_{\p X} - N(B),
\end{equation*}
where $X$ is the closure of $M\setminus N(B)$. Here $X$ is a spin
manifold with boundary $\p X=\p N(B)$ and we have a principal
$S^1$-bundle $\p X=\p N(B)\to B$.  In particular, the boundary 
$\p N(B)$ has a natural free $S^1$-action, and $N(B)$ can be identified
with the unit disk bundle of $L$.
\begin{definition}
  \label{def:adaptedmetric}
  A Riemannian metric $g$ on a $\spinc$-manifold $(M,L)$ is called
  \emph{well-adapted} if the restriction $g|_{N(B)}$ is
  $S^1$-invariant (where $M = X\cup_{\p X} - N(B)$) and if the
  the metric on $X$ is a product metric in a collar neighborhood
  of $\p X$.
\end{definition}
The case of manifolds $M$ which are not spin but which have a double
cover which is spin is also closely analogous.  In this case,
the double cover is classified by a map $c\co M \to  \bR\bP^{\infty}$
which we can take to land in some $\bR\bP^k\subset \bR\bP^{\infty}$
and to be transverse to $\bR\bP^{k-1}$,
and $c$ defines a real line bundle $L$ over $M$.  This time there is
a submanifold $B$ of $M$ of codimension $1$ and a decomposition of
$M$ as $M = X\cup_{\p X} N(B)$, and the submanifold $B$ is dual to
the real line bundle $L$.  In this case we have a principal
$C_2$-bundle $\p X=\p N(B)\to B$, with $C_2=\{\pm 1\}$, and $N(B)$ is
again the unit disk bundle of $L$ (except that the ``disks'' are
one-dimensional, and can be identified with $[-1, 1]$). The parallel
to Definition \ref{def:adaptedmetric} is:
\begin{definition}
  \label{def:adaptedmetric1}
  A Riemannian metric $g$ on a nonspin manifold $M$ with spin double
  cover and associated real line bundle $L$ as above is
  \emph{well-adapted} if the restriction $g|_{N(B)}$ is
  $C_2$-invariant (where $M = X\cup_{\p X} N(B)$) and if the
  metric on $X$ is a product metric in a collar neighborhood
  of $\p X$.
\end{definition}

Here is the question we are interested in:
\begin{question}
\label{q:adpatedpsc}  
  Suppose we have $(M,L)$, where $L$ is a complex or real line
  bundle on $M$, a nonspin manifold which is $\spinc$ in the
  complex case or has a spin double cover in the real case. Under what
  conditions does there exist a well-adapted psc-metric $g$ on $(M,L)$?
\end{question}
We will study this under the simplifying assumptions that
  $X$ and $\p X$ are simply connected.

\subsection{Plan}
The plan for this survey is fairly straightforward.  In Section
\ref{sec:pin}, we discuss the case of Question \ref{q:adpatedpsc}
when $M$ has a spin double cover.  This actually involves two
separate questions.  The first is more basic: if $M$ is closed and nonspin
but has a simply connected spin double cover, what is the necessary
and sufficient condition for $M$ to admit a Riemannian metric
of positive scalar curvature?  We answer this question, some special
cases of which had previously been treated in \cite{MR1688827,MR1694601},
completely in dimensions $5$ and up.  Then we go on to ask the
more refined question of when such a psc-metric on $M$ can be taken
to be well-adapted.  The main result on this, which answers the
question completely, is Theorem \ref{thm:mainC2}.

In Section \ref{sec:spinc}, we discuss the $\spinc$ case
of Question \ref{q:adpatedpsc}.  This involves several new
considerations: a new application of $\spinc$ index theory,
an analysis of a twisted version of positive scalar curvature,
where the scalar curvature is perturbed by the curvature of a line
bundle, and a study of new transfer map involving $\bC\bP^2$-bundles.
This latter map had been studied in part before by F{\"u}hring
\cite{Fuhring2}, but with a different application in mind.
The most interesting results of this section are
Theorem \ref{cor:nonspinpsc2} and Theorem \ref{spinc-well-adap}.

In the final section, Section \ref{sec:sing}, we discuss
a more general framework of ``manifolds with singularities''
(really, compact pseudomanifolds with two strata), and when they
admit well-adapted psc-metrics.  This section is based on
recent joint work of the authors with Paolo Piazza.

\section{$C_2$-Bundles and $\mathbf{Pin}^\pm$-Manifolds}
\label{sec:pin}

We begin with the case of $C_2$-bundles, which is slightly less
complicated than the case of $S^1$-bundles.  In this section we will
be interested in the study of manifolds $M^n$
which are not spin, but which have a spin double cover $\tM$,
which we will assume for simplicity to be simply connected.

The orthogonal group $\mathrm{O}(n)$ has two connected components, the
connected component of the identity being $\SO(n)$.  The double cover
of $\SO(n)$ (which is also the universal cover when $n\ge 3$) is $\Spin(n)$.
But there are two ways to complete the commuting diagram of group
extensions
\[
\xymatrix{1\ar@{=}[d]\ar[r]& \bZ/2\ar@{=}[d]\ar[r] &\Spin(n)
  \ar@{^{(}->}[d]\ar@{->>}[r] &\SO(n)\ar@{^{(}->}[d]\ar[r] & 1\ar@{=}[d]\\
  1\ar[r]& \bZ/2\ar[r] & G
  \ar@{->>}[r] &\mathrm{O}(n)\ar[r] & \,1.}
\]
(This can be seen, for example, using the inflation-restriction
sequence in cohomology of the split group extension
\[
\xymatrix{1\ar[r] &\SO(n) \ar[r]& \mathrm{O}(n)
  \ar[r] &C_2 \ar@/_/[l]\ar[r] & 1}
\]
with Borel cochains for the coefficient
group $\bZ/2$ \cite[Chapter I]{MR171880}.)
The two possibilities for $G$ are called $\Pin^+(n)$ and $\Pin^-(n)$
(they are described explicitly in terms of Clifford algebras
in \cite[Chapter I, \S2]{lawson89:_spin}),
and a lifting of the orthogonal frame bundle to a principal bundle
for $\Pin^\pm$ is called a pin$^\pm$ structure on a manifold
\cite{Pin}. Such structures can exist even when a manifold is non-orientable,
and they still make it possible to do some aspects of spinor geometry
without an orientation.

Suppose $M^n$ is not spin and has a simply connected double cover $\tM$.
Then up to homotopy, we have a fibration
$\tM\xrightarrow{p} M\xrightarrow{c} \bR\bP^\infty$, where $c$ is the
classifying map of the covering map $p$. Since $\tM$ is simply
connected, the Serre spectral sequence (if $\pi_1(\bR\bP^\infty)\cong C_2$
acts trivially on the low dimensional
cohomology of $\tM$) gives an exact sequence
\begin{equation}
\label{eq:Serre}  
\begin{array}{l}
  \!\!\!\!
  0\!\to \!H^2(\bR\bP^\infty,\!\bZ/2)\!\xrightarrow{c^*}\! H^2(M,\!\bZ/2)\!
\xrightarrow{p^*}\! H^2(\tM,\!\bZ/2)\! \xrightarrow{d_3}\!
H^3(\bR\bP^\infty,\bZ/2)
\end{array}
\end{equation}
\begin{remark}
\label{rem:pinvariants}  
There are now various subcases to consider:
\begin{enumerate}
\item $M$ is orientable, so that $w_1(M)=0$.  Since we are assuming that
  $M$ is not spin, we have $w_2(M)\ne 0$, while $p^*w_2(M)=w_2(\tM)=0$.
  Here $c$ is the classifying map for the universal cover of $M$, and
  by \eqref{eq:Serre}, $w_2(M)$ is pulled back from the
  generator of $H^2(\bR\bP^\infty,\bZ/2)$ under $c^*$.
\item $M$ is not orientable, so that $w_1(M)\ne 0$ and the generator of
  $$H^1(\bR\bP^\infty,\bZ/2)$$ pulls back under $c^*$ to $w_1(M)$. Since
  $p^*w_2(M)=w_2(\tM)=0$, either $w_2(M)=0$, in which case $M$ admits
  a pin$^+$-structure, or else $w_2(M)\ne 0$ but $w_2(M)$ comes
  from $M\xrightarrow{c} \bR\bP^\infty$, i.e., $w_2(M)=w_1(M)^2$,
  in which case $M$ admits a pin$^-$-structure \cite{MR1171915,MR1069818,Pin}.
  pin$^-$ and pin$^+$-structures are also known (cf.
  \cite{MR261613,MR321123}) as Pin and Pin$'$ structures, respectively.
\end{enumerate}
The various cases are illustrated by real projective spaces $\bR\bP^n$.
These are orientable for $n$ odd and non-orientable for $n$ even.
$\bR\bP^n$ is spin when $n\equiv3\mod 4$, orientable but non-spin
when $n\equiv1\mod 4$, pin$^+$ when $n\equiv0\mod 4$,
pin$^-$ when $n\equiv2\mod 4$.
\end{remark}

The first step is to study obstructions to psc-metrics in these
various cases.  Here is the basic result.  This fact was
conjectured before \cite[Conjecture 2.3.1]{MR3999671}
but we don't know of a complete proof in the literature,
although one special case (the pin$^-$-case with $n\equiv2\mod 4$)
is in \cite[Theorem 6.3]{MR1694601}.
\begin{theorem}
\label{thm:pin}
Let $M^n$ be a closed pin$^+$ or pin$^-$ manifold with
fundamental group $\bZ/2$ and spin universal cover $\tM$. Fix a Riemannian
metric on $M$ and let $L$ be the determinant line bundle of the orthogonal
frame bundle.  {\lp}So $w_1(L)=w_1(M)$.{\rp} In the pin$^-$ case,
$TM\oplus L$ admits a spin structure, and we obtain an associated
$\Cl_{n+1}$-linear Dirac operator $\Dirac_-$
with index $\alpha^-(M)$ in $KO_{n+1}$. In the
pin$^+$ case, $TM\oplus L\oplus L\oplus L$ admits a spin structure, and we
obtain an associated $\Cl_{n+3}$-linear Dirac operator $\Dirac_+$ with index
$\alpha^+(M)$ in $KO_{n+3}$. These indices are obstructions to existence
of a psc-metric on $M$.

Assume that $n\ge 5$, that $\alpha(\tM)=0$ in $KO_n$ and that
$\alpha^-(M)=0$ in $KO_{n+1}$ in the pin$^-$ case, and that
$\alpha^+(M)=0$ in $KO_{n+3}$ in the pin$^+$ case
Then $M$ admits a metric of positive scalar curvature.
\end{theorem}
\begin{proof}
  Since $L$ has a canonical flat connection, the Lichnerowicz identity
  $\Dirac_\pm^2=\nabla^*\nabla + \frac{\kappa}{4}$, where $\kappa$ is
  the scalar curvature, holds just as for the Dirac operator on a spin
  manifold, and so if $\kappa>0$, the spectrum of $\Dirac_\pm$ is bounded
  away from $0$ and the indices $\alpha^\pm(M)$ must vanish. Similarly,
  if $M$ admits a psc-metric, then so does $\tM$, and so $\alpha(\tM)=0$.
  This takes care of the necessity.

  For sufficiency, we use the Bordism Theorem \cite[Theorem 2.1.1]{MR3999671},
  which in this case boils down to the statement that it is sufficient
  to check that every class in $\Omega_n^{\Pin^\pm}$, $n\ge 5$, for which the
  necessary conditions hold has a representative of {\psc}.  The relevant
  bordism groups might as well be localized at $2$ since by
  \cite[Corollary 4 to Theorem 3]{MR1069818},
  the group $\Omega_*^{\Pin^\pm}$ are $2$-primary torsion, and divided
  by the $\Omega_*^{\Spin}$-submodules generated by even real projective
  spaces (which are obviously represented by manifolds of {\psc})
  are $\bF_2$-vector spaces.

  We begin with the case of
  $\Omega_n^{\Pin^-}\cong \widetilde\Omega_{n+1}^{\Spin}(\bR\bP^\infty)$, see
  \cite{MR261613}. Localized at $2$, there is an additive splitting
  $\MSpin\simeq \ko\vee (\image T)$, where $T$ is the transfer
  $$\MSpin\wedge \Sigma^8 B\PSp(3)_+\to \MSpin$$ defined by the
  $\bH\bP^2$-bundle construction, see \cite{MR1189863,MR1259520}. Classes
  in $\Omega_n^{\Pin^-}$  coming from the homotopy of
  $(\image T)\wedge \bR\bP^\infty$ are linear combinations of classes
  represented geometrically by dualizing a real line bundle $L$
  on a manifold with a bundle structure $\bH\bP^2\to P^{n+1}\to N^{n-7}$,
  where the structure group of the bundle is the isometry group of
  $\bH\bP^2$.  The line bundle $L$ has to come from $N$ since
  $\bH\bP^2$ is simply connected, so the resulting pin$^-$ $n$-manifold
  has to be an $\bH\bP^2$-bundle over a pin$^-$ $(n-8)$-manifold
  obtained by dualizing a real line bundle over $N$,
  and thus has a psc-metric.  So we are reduced to looking at
  $\pi_*(\ko\wedge \bR\bP^\infty)$.  These groups were
  computed in \cite[Theorem 5.1]{MR261613} and in
  \cite[Theorem 1]{MR1069818}.  There are summands of $\bZ/2$ in
  dimensions $0$ and $1$ mod $8$ that are detected by the index invariant
  $\alpha^-$, so these groups are \emph{not} represented by
  manifolds of {\psc}.  The rest of the homotopy groups of
  $\ko\wedge \bR\bP^\infty$ correspond to cyclic summands generated by
  real projective spaces of dimension $2$ mod $4$, and the result is
  obviously true for these.  So that completes the proof for the pin$^-$
  case.  The  pin$^+$ case is quite similar using the equivalence
  $$\MPin^+\simeq \MSpin\wedge M(3)$$ of \cite[Theorem 1]{MR1069818}.
  Everything is the same except for replacement of $M(1)$ by $M(3)$.
  The homotopy groups of the summand $(\image T)\wedge M(3)$
  are again represented by $\bH\bP^2$-bundles over lower-dimensional
  pin$^+$ manifolds, while the homotopy groups $\pi_*(\ko\wedge M(3))$
  include cyclic summands generated by
  real projective spaces of dimension $0$ mod $4$, as well as copies
  of $\bZ/2$ in dimensions $2$ and $3$ mod $8$ that are detected by
  the index invariant $\alpha^+$.  So again the theorem is true.
\end{proof}

We now want to answer question \ref{q:adpatedpsc} in the case where
$\tM$, $X$, and $\p X$ are simply connected.  Since we have a
double covering $p\co \p X\to B$, $\pi_1(B)\cong \pi_1 N(B)\cong \pi_1(M)$.
Now existence of a well-adapted psc-metric on $M$ implies
existence of a psc-metric on $B$. There are a number of distinct cases:
\begin{remark}
\label{rem:indobstrdoublecover}
Let $B$ have a simply connected spin double cover. 
\begin{enumerate}
\item If $B$ is spin, then since the Gromov-Lawson-Rosenberg conjecture
  holds for the fundamental group $\bZ/2$ \cite{MR1133900,MR1484887},
  if $B$ admits a psc-metric,
  then $\beta\circ c_*([B])=0$ in $KO_{n-1}(\bR[C_2])=KO_{n-1}\oplus KO_{n-1}$,
  and the converse holds if $n\ge 6$ (so that $\dim B\ge 5$). Here
  $\beta$ is the $KO$-assembly map discussed in \cite{MR1133900}.
\item If $B$ is oriented but non-spin, then $B$ admits a Dirac operator
  twisted by a direct sum of two copies the real line bundle defined
  by $c$ (this bundle has nontrivial $w_2$), but not an untwisted
  Dirac operator.  The index of this twisted Dirac operator is an
  obstruction to psc in $KO_{n-1}$, and a variant of Stolz's Theorem
  \cite{MR1189863} shows that vanishing of this obstruction is sufficient
  for $B$ to admit a psc-metric if $n-1\ge 5$ or $n\ge 6$
  \cite{MR1688827}.
\item If $B$ is not orientable but has a pin$^\pm$ structure, then
  the obstructions to a psc-metric on $B$ are covered
  by Theorem \ref{thm:pin} above.
\end{enumerate}
\end{remark}

Now we are ready for the main theorem of this section.
\begin{theorem}
\label{thm:mainC2}
Let $M^n$ be a non-spin closed $n$-manifold with a simply connected
spin double cover $\tM$.  Write $M = X\cup_{\p X} N(B)$ as above,
with $X$ and $\p X$ simply connected.  If $M$ admits a well-adapted
psc-metric, then $\alpha(\tM)=0$ in $KO_n$, there are additional
index obstructions to a psc-metric on $M$ enumerated in
\textup{Remark \ref{rem:indobstrdoublecover}} above,
  and $B$ admits a psc-metric. We thus also have to have
the vanishing of an index obstruction for $B$, which depends
on what subcase of \textup{Remark \ref{rem:pinvariants}} applies to $B$.
The converse, i.e., the statement that the vanishing of all
these index obstructions {\lp}for both $M$ and $B${\rp}
implies the existence of a well-adapted
psc-metric on $M$, holds if $n\ge 6$.
\end{theorem}
\begin{proof}
  Let $\beta M=B$ be the ``Bockstein,'' the quotient of $\p X$ by the
  free $C_2$-action.  (The reason for the name ``Bockstein'' will appear
  later.)  First suppose that $\beta M$ is spin.  Since $N(\beta M)$
  is the disk bundle of the flat line bundle $L$ defined by the spin double
  cover of $M$, restricted to $\beta M$ where the nontriviality of the
  bundle is concentrated, the tangent bundle of $M$, restricted to
  $N(\beta M)$, is $p^*T(\beta M)\oplus L$, and since the tangent bundle
  of $\beta M$ is spin, we have $w_2(M)=0$ (since $X$ is also spin)
  and $w_1(M)=w_1(L)\ne 0$.  So $M$ is a pin$^+$-manifold.  (A case to
  keep in mind is $X = D^n$, $\p X=S^{n-1}$, $\beta M = \bR\bP^{n-1}$,
  $M=\bR\bP^n$, in the case where $n\equiv 0\mod 4$.) So we have four
  obstructions to a well-adapted psc-metric on $M$: $\alpha(\tM)\in KO_n$,
  $\alpha^+(M)\in KO_{n+3}$ from Theorem \ref{thm:pin}, and
  the two $KO_{n-1}$-valued obstructions to a psc-metric on $\beta M$.
  (The last two are the indices of the untwisted and twisted Dirac
  operators on $\beta M$.) Except for $\alpha(\tM)$ when $n\equiv0\mod 4$,
  all of these obstructions are $\bZ/2$-valued.  
  We have the following long exact sequence:
  \begin{equation}
  \label{eq:spinandpinplus}  
  \cdots \xrightarrow{\delta} \Omega_n^{\Spin} \xrightarrow{f}
  \Omega_n^{\Pin^+} \xrightarrow{\beta} \Omega_{n-1}^{\Spin}(BC_2)
  \xrightarrow{\delta} \Omega_{n-1}^{\Spin} \to \cdots.
\end{equation}
  Here $f$ is the forgetful map that forgets that a spin manifold is
  oriented and considers it as a pin$^+$ manifold via the natural map
  of Lie groups $\Spin(n)\hookrightarrow \Pin^+(n)$.  The ``Bockstein''
  map $\beta$ dualizes the line bundle on a pin$^+$ manifold defined
  by $w_1$, and produces a spin manifold one dimension lower.
  We call $\beta$ the Bockstein since it is a
  dimension-shifting connecting map in this exact sequence, just like
  the classical Bockstein map for homology. 
  The transfer map $\delta$ takes a spin manifold with a map to $BC_2$
  to the associated double cover of the manifold.  The exact sequence
  \eqref{eq:spinandpinplus} may be derived from a related exact sequence
  \[
  \cdots \to \pi_n(\MSpin\wedge\bZ/2)\to \Omega_n^{\Pin^+}
  \to \Omega_{n-2}^{\Pin^-}\to \pi_{n-1}(\MSpin\wedge\bZ/2)\to \cdots
  \]
  in \cite[Lemma 7]{MR1069818}, since $\Omega_{n-1}^{\Spin}(BC_2)$
  splits as $\Omega_{n-2}^{\Pin^-}\oplus \Omega_{n-1}^{\Spin}$,
  and $\delta$ restricted to the second factor is just multiplication
  by $2$ from $\Omega_{n-1}^{\Spin}$ to itself, which gives rise to
  a cofiber of $\pi_{n-1}(\MSpin\wedge\bZ/2)$.

  Now suppose that $n\ge 6$ and that all of the index
  obstructions vanish. We will use \eqref{eq:spinandpinplus}
  to show that $M$ admits a well-adapted psc-metric.
  By the Gromov-Lawson-Rosenberg
  conjecture for the group $\bZ/2$, $\beta M$ admits a psc-metric.
  Lift this to a local product metric on $N(\beta M)$ (which is locally
  the product of $\beta M$ with a flat real line bundle); this gives
  a $C_2$-invariant psc-metric on $\p X = \p(N(\beta M))$ which is
  a product metric on the boundary.
  The double $P$ of $N(\beta M)$ along $\p X$ is a pin$^+$-manifold
  admitting a psc-metric which agrees with $M$ on $N(\beta M)$.
  By the sequence \eqref{eq:spinandpinplus}, $M$ is pin$^+$-bordant
  to the disjoint union of $P$ and a closed spin manifold $M'$.  By the
  additivity of the index invariants, the index invariants for $M'$
  vanish, so $M'$ admits a psc-metric.  Thus $M'\sqcup P$ is a manifold
  with a psc-metric in the same pin$^+$-bordism class as $M$.  We can carry
  the psc-metric across the bordism, and since the necessary surgery can be
  done on the interior of $X$, away from $\p X$, without changing 
  $N(\beta M)$, the resulting psc-metric on $M$ is well-adapted.  This
  takes care of the case where $\beta M$ is a spin manifold.

  If $\beta M$ is oriented but not spin (case (2) of Remark
  \ref{rem:indobstrdoublecover}, the case studied in \cite{MR1688827},
  like the case of $X = D^n$, $\p X=S^{n-1}$, $\beta M = \bR\bP^{n-1}$,
  $M=\bR\bP^n$, in the case where $n\equiv 2\mod 4$),
  things are quite similar except that
  $M$ is now a pin$^-$-manifold instead of a pin$^+$-manifold.  The
  replacement for \eqref{eq:spinandpinplus} in this case is the following:
  \begin{equation}
  \label{eq:twspinandpinminus}  
  \cdots \xrightarrow{\delta} \Omega_n^{\Spin} \xrightarrow{f}
  \Omega_n^{\Pin^-} \xrightarrow{\beta} \Omega_{n-1}^{\Spin,\text{tw}}
  \xrightarrow{\delta} \Omega_{n-1}^{\Spin} \to \cdots.
  \end{equation}
  Here $\Omega_{n-1}^{\Spin,\text{tw}}$ is the bordism group in dimension $n-1$
  for oriented manifold with a spin double cover for a given
  element of $H^1(-,\,\bZ/2)$, $\delta$ again corresponds to the
  double cover, and $f$ again is the forgetful map, this time corresponding
  to the inclusion  $\Spin(n)\hookrightarrow \Pin^-(n)$. We have a
  Bockstein map $\beta$ as before.  In this case $N(\beta M)$, and hence
  $M$, has a pin$^-$ structure since $w_2$ is pulled back from $\bR\bP^\infty$
  via the map associated to $w_1$, or in other words, $w_2(M)=w_1(M)^2$.
  With the substitution of \eqref{eq:twspinandpinminus} for
  \eqref{eq:spinandpinplus}, the proof works just as before.

  The final case involves a related long exact sequence of
  \cite[Theorem 3.1]{MR321123}, relevant to the case where $M$ is an
  oriented non-spin manifold and $\beta M$ is a pin$^+$ manifold.
  This situation
  arises when $X = D^n$, $\p X=S^{n-1}$, $\beta M = \bR\bP^{n-1}$,
  $M=\bR\bP^n$, and $n\equiv 1\mod 4$.

  After renaming the maps from what they are called in \cite{MR321123}, 
  this sequence is as follows:
  \begin{equation}
    \label{eq:Giambseq}
  \cdots \xrightarrow{\delta} \Omega_n^{\Spin} \xrightarrow{\mathrm{enh}}
  \Lambda_n \xrightarrow{\beta} \Omega_{n-1}^{\Pin^+}
  \xrightarrow{\delta} \Omega_{n-1}^{\Spin} \to \cdots.
  \end{equation}
  Here $\Lambda_n$, defined in \cite{MR321121}, is a bordism group
  of pairs $(M,L)$ where $M$ is an oriented manifold, $L$ is a real
  line bundle, and $w_2(M)=w_1(L)^2$.
  Again $\beta$ is the \emph{Bockstein map}, sending the class of
  $(M,L)$ to the class of $\beta M$, dual to the line bundle $L$,
  $\delta$ is a transfer map, taking
  the class of a pin$^+$-manifold to the class of its spin double cover,
  and ${\mathrm{enh}}$ is an \emph{enhancement map}, sending the class
  of a spin manifold to the class of the same manifold paired with the
  trivial line bundle.

  We use \eqref{eq:Giambseq} as follows.  Suppose $M= X\cup_{\p X}N(\beta M)$
  is oriented but non-spin, with a spin double cover.  There is a real
  line bundle $L$ associated to the spin double cover, and $w_2(M)=w_1(L)^2$,
  so $(M,L)$ gives a class in $\Lambda_n$.  Now suppose $n\ge 6$ and all the
  index invariants of $M$ and $\beta M$ vanish.  Applying
  Theorem \ref{thm:pin} to $\beta M$, we see that it admits a psc-metric.
  Lift this to a local product metric on $N(\beta M)$.  As before, $M$
  is bordant (in the sense of the theory $\Lambda$) to $(M'\sqcup P, L)$,
  where $P$ is the double of $N(\beta M)$, $L$ lives on $P$ and is trivial
  on $M'$, and $M'$ is a closed spin manifold.  Again, additivity of
  the index invariants implies that $M'$ admits a psc-metric.  Then we
  use the bordism method to transfer the psc-metric from $M'\sqcup P$ to
  $M$, doing the surgery away from $N(\beta M)$, so that the metric we
  get is well-adapted.  That concludes the proof.
\end{proof}

\section{$S^1$-Bundles and $\mathbf{Spin}^c$-manifolds}
\label{sec:spinc}
\subsection{Preliminary observations and examples}
Let $(M,L)$ be a non-spin
$\spin^c$-manifold. We choose a submanifold $B\subset M$
dual to the bundle $L$; in particular, we identify the restriction
$L|_B$ with the normal bundle of the embedding $B\hookrightarrow M$.
Let $N(B)$ be a tubular neighborhood of $B$; we denote by $X$ the
closure of $M\setminus N(B)$. 

We obtain a decomposition $M = X\cup_{\p X} -N(B)$. Here $X$ is a spin
manifold, whose boundary $\p X$ is equipped with free $S^1$-action
since $\p X$ is the total space of the circle bundle $\p X\to B$.  This
action is consistent with a natural $S^1$-action on the tubular
neighborhood $N(B)$ since $N(B)$ is the disk bundle of the restriction
$L|_B$. 
\begin{remark}
Let $M$ be a non-spin simply connected $\spin^c$ manifold, i.e., with
$w_2(M)\ne 0$. The projective space $\bC\bP^n$ is an example of such $M$
if $n$ is even. Then a $\spin^c$-structure is given by a complex line
bundle $L$ on $M$ such that $c_1(L)$ reduces mod $2$ to
$w_2(M)$. Then, as we discussed above, $B$ is dual to $L$; by
construction, $c_1(L)$, and thus also $w_2(X)$, is trivial on the
complement $X$ of a tubular neighborhood $N(B)$ of $B$. Thus $X$ is a
spin manifold with boundary $\p X$, which is a circle bundle over
$B$.  We notice that the manifold $B$ is spin, since
\begin{equation*}
w_2(B) + \bigl(c_1(L)\!\!\!\mod 2\bigr) =
w_2\left(\left.TM\right\vert_{B}\right) =
\iota^* w_2(M) = \bigl(c_1(L)\!\!\!\mod 2\bigr),
\end{equation*}
which says that $w_2(B)=0$. 
\end{remark}
Since $B\subset M$ is a spin manifold, we have $\alpha(B)\in KO_{n-2}$
which evaluates the index of the Dirac operator
$\Dirac_{B}$.  Let $\Omega^{\spinc}_n$ be the $\spinc$-bordism group.
We also have a natural homomorphism
\begin{equation}\label{eq:alphac}
\alpha^{\spinc}\co \Omega^{\spinc}_n\to KU_n
\end{equation}
which  evaluates the index of the $\spinc$ Dirac operator
\begin{equation*}
\alpha^{\spinc}\co [(M,L)] \mapsto [\Dirac_{(M,L)}]\in KU_n.
\end{equation*}
Recall that a well-adapted metric $g$ on $M$ is such a
Riemannian metric that the restriction $g|_{N(B)}$ is $S^1$-invariant
and $g$ is a product-metric near $\p X = - \p N(B)$.

The following geometrical result gives a necessary condition for
existence of a well-adapted psc-metric:
\begin{theorem}[{\cite[Theorem C]{BB}}]
\label{thm:BB}
Let $Z$ be a compact manifold with free
$S^1$-action. Then $Z$ admits an $S^1$-invariant
psc-metric if and only if the quotient manifold $B=Z/S^1$ admits a psc-metric.
\end{theorem}
\begin{example}
\label{ex:K3}
(i) Let $B$ be a $K3$-surface (which is a simply connected spin
$4$-manifold with nonzero $\widehat A$-genus). Then $B$ does not admit
a psc-metric, but there is a circle bundle $p\co Y\to B$ with 
simply connected total space $Y$. To construct such a bundle, we
choose a primitive element $c\in H^2(B,\bZ)\cong \bZ^{22}$ and find a
complex line bundle $L(c)$ with $c_1(L(c))=c$. Then the bundle
$p\co Y\to B$ is the circle bundle of the line bundle $L(c)$.

The manifold $Y$ is necessarily spin, since $T_{Y}\cong p^*T_{B}
\oplus V$, where $V$ is the real tangent line bundle along the circle
fibers, which is trivial, and thus $w_2(T_{Y}) = p^*w_2(T_{B}) = 0$.
Furthermore, $Y$ is a spin boundary, since
$\Omega^\spin_5=0$. Thus there is a spin $6$-manifold $X$ with $\p X=Y$,
and we can do surgery on $X$ away from the boundary to ensure that $X$
is simply connected and
the pair $(X,Y)$ is $2$-connected.  In particular, we obtain that $Y$
has a psc-metric $g_Y$. However, Theorem \ref{thm:BB} implies
that any such psc-metric $g_Y$ cannot be $S^1$-invariant since
otherwise it would give a psc-metric on $B=K3$, which is not possible.

To construct a relevant $\spin^c$-manifold, we glue together $X$ and the
disk bundle $N(B)$ of $L(c)$ over $B$. With a little bit of work one can show
that non-spin $\spin^c$-manifold $M:=X\cup_{\p X}-N(B)$ comes together
with a line bundle $L\to M$ dual to $B$ (and trivial over $X$), so
that $L|_{B}$ coincides with the bundle $L(c)$ we started
with. Then the manifold $M$, being non-spin and simply-connected,
admits a psc-metric $g_M$. Thus we conclude that any psc-metric $g_M$
on $(M,L)$ cannot be well-adapted, since otherwise we would obtain an
$S^1$-invariant psc-metric on $Y$, and, consequently, a psc-metric on
$B=K3$ by Theorem \ref{thm:BB}.
\vspace{1mm}

(ii) Let $\Sigma^{10}$ be a homotopy $10$-sphere with nonzero
$\alpha$-invariant (i.e., representing the generator of
$KO_{10}=\bZ_2$). We consider the $\spin$-manifold
$B= \Sigma^{10}\# \bC\bP^5$. Since the $\alpha$-invariant is additive
on connected sums, the manifold $B$ does not admit a
psc-metric. Notice that $B$ is a \emph{fake complex projective space},
so it admits a principal $S^1$-bundle $Y\to B$ for which the total
space $Y$ is a homotopy $11$-sphere. There being no torsion in
$\Omega^\spin_{11}$, the exotic sphere $Y$ is a spin boundary and we
can choose a spin $12$-manifold $X$ bounding $Y$, such that $(X,Y)$ is
2-connected. Just as in the example (i), we construct a non-spin
$\spin^c$-manifold $(M,L)$ with $M=X\cup_{\p X} -N(B)$ (and $L$ dual
to $B$) which admits a psc metric $g_M$, but no such psc-metric is
well-adapted, since otherwise it would produce a psc-metric on $B$
(again, via Theorem \ref{thm:BB}).
\end{example}
These examples show that existence of a well-adapted psc-metric on a
$\spin^c$-manifold $(M,L)$ implies that the manifold $B$ dual to $L$ has
to admit a psc-metric. Since $B$ is spin, we obtain the first
obstruction $\alpha(B)\in KO_{n-2}$ for existence of a well-adapted
psc-metric on $(M,L)$. In the case when the manifold $B$ is
simply-connected (and $n-2\geq 5$), this is the only obstruction for
existence of a psc-metric on $B$.

Next, we choose a psc-metric $g_{B}$ on $B$. Then it gives us an
$S^1$-invariant psc-metric on $N(B)$ and, in particular, a psc-metric
$g_Y$ on $Y=\p N(B)$. Let $g_M$ be some well-adapted metric on $(M,L)$
(which is not necessarily a psc-metric outside of $N(B)$) extending the
above metric on $N(B)$.  Then to construct a well-adapted psc-metric
on $M$, it is enough to extend $g_Y$ to a psc-metric $g_X$ such that
$g_X$ is a product-metric near $Y$: we keep in mind that
$M=X\cup_{\p X=Y}-N(B)$.

Consider the $\spinc$ Dirac operator $\Dirac_{(M,L)}$ on $(M,L)$. 
The operator $\Dirac_{(M,L)}$ depends on a choice of a connection
$A_L$ on the line bundle $L$;
however, since the restriction $L|_X$ is trivial, we can choose
the connection $A_L$ to be flat on $L|_X$. Thus we can take the
restriction $\Dirac_{(M,L)}|_{X}$ to be the usual $\spin$ Dirac
operator. Assuming that $g_M$ restricts to a psc-metric $g_Y$, we
obtain a proper APS\footnote{Atiyah-Patodi-Singer}-boundary
problem for the Dirac operator on
$(X,Y,g_Y)$. We denote by $\Dirac_{(X,Y,g_Y)}$ the resulting Dirac
operator.

We obtain the next obstruction for existence a well-adapted psc-metric
$g_M$ on $(M,L)$, given by the relative index
$\alpha^{\mathrm{rel}}(X,Y, g_Y)\in KO_n$ of the operator
$\Dirac_{(X,Y,g_Y)}$.  Notice that \emph{a priori} the index
$\alpha^{\mathrm{rel}}(X,Y, g_Y)$ depends on a choice of $g_Y$, and,
consequently, on a choice of a psc-metric $g_B$.

As before we fix a connection $A_L$ on the line bundle $L$ which is flat on
$L|_X$, and consider again the Dirac operator $\Dirac_{(M,L)}$ on
$(M,L)$. We have the Lichnerowicz formula:
\begin{equation}\label{eq:Lich}
\Dirac_{(M,L)}^2 = \nabla^*\nabla
+ \frac14 s_{g_M} + \mathcal{R}_{A_L},
\end{equation}
where $g_M$ is a well-adapted psc-metric and $\mathcal{R}_{A_L}$ is a
corresponding curvature form. Then we can homotope the metric on the
fibers of $N(B)\to B$ to make it equal to that on a round hemisphere
$S^2_+\subset S^2(r)$ with the hemispherical fibers having small diameter
$r$ and thus big curvature. That implies we can make $s_{g_M}$ highly
positive without changing the curvature term $\mathcal{R}_{A_L}$. This
allows us to bound the square of the
Dirac operator $\Dirac_{(M,L)}^2$ away from
$0$.  Thus $\alpha^{\spinc}(M,L)=0$, where
$\alpha^{\spinc}\co \Omega^{\spin^c}_n\to KU_n$ is the index map.

We conclude: \emph{a priori} there are three obstructions for existence of a
well-adapted psc-metric on a $\spinc$-manifold $(M,L)$, where
$M=X\cup_{\p X=Y} -N(B)$ and $B$ is dual to $L$ as above:
\begin{enumerate}
\item[(i)] $\alpha(B)\in KO_{n-2}$;
\item[(ii)] $\alpha^{\mathrm{rel}}(X,Y, g_Y)\in KO_n$;
\item[(iii)] $\alpha^{\spinc}(M,L)\in KU_n$.
\end{enumerate}
We emphasize that the obstructions $\alpha(B)$ and
$\alpha^{\spinc}(M,L)\in KU_n$ are primary obstructions, and the
obstruction $\alpha^{\mathrm{rel}}(X,Y, g_Y)$ is secondary one: it
depends on a choice of a psc-metric $g_B$ on $B$.

\emph{A priori}, it is not clear why vanishing of these three obstructions
should imply existence of a well-adapted psc-metric, especially when it
comes to the secondary obstruction;
at least we do not know how to use this information
directly to construct such a metric. First, we would like to describe
geometrical meaning of the obstruction $\alpha^{\spinc}(M,L)\in KU_n$.

Before we get to this, however, we should point out that contrary to
what it might seem from the above, the obstructions (i) and (iii)
on our list are actually not independent.  By
\cite[Theorem 3]{MR2092777}, proved independently in
\cite{MR1245100,MR2514364}, $\alpha^{\spinc}(M,L)$ determines
$\alpha(B)$, in the following sense: when both are integers
(which happens when $n\equiv2\mod 4$, so that $KO_{n-2}\cong \bZ$),
they are equal, and when $\alpha^{\spinc}(M,L)$ is an integer but
$\alpha(B)$ is an integer mod $2$ (which happens when $n\equiv4\mod 8$),
then $\alpha(B)$ is the mod $2$ reduction of $\alpha^{\spinc}(M,L)$.
The three proofs of these facts (one by Ochanine and Fast and two
by Weiping Zhang) are very interesting exercises either in
index theory or in bordism theory, but would take us away from
our main theme here.  However, let us point out an interesting
application to Example 1.12.  In the first part of this example,
we constructed a case when $M$ is a $\spinc$ $6$-manifold and $B$ is a spin
manifold with non-zero $\widehat A$-genus.  Thus the theorem says
that whatever choice we take for $M$, it has to satisfy
$\alpha^c(M)=\widehat A(B)$.  In the second part of this example,
we constructed a  $\spinc$ $12$-manifold $M$, where $B$
is a homotopy $\bC\bP^5$ with nonzero $\alpha$-invariant
in $KO_{10}\cong \bZ/2$.  In this case, the theorem says that
whatever choice we take for $M$, $\alpha^c(M)$ has to be odd.
\subsection{Geometry of the index $\alpha^{\spinc}$}
Let $(M,L)$ be a non-spin $\spinc$-manifold. We choose a metric $g$ on
$M$, a hermitian metric $h$ on $L$, and a (unitary) connection
$A_L$ on $L$. These data give
us the $\spinc$ Dirac operator $\Dirac_{(M,L)}$. We have the
Lichnerowicz formula
\begin{equation*}
\Dirac_{(M,L)}^2 = \nabla^*\nabla + \frac14 \kappa_{g}+ \mathcal{R}_L
\end{equation*}
where the term $\mathcal{R}_L$ has the following form:
\begin{equation*}
 \mathcal{R}_L =\frac12\sum_{j<k}F_L(e_j,e_k)\cdot e_j\cdot e_k
\end{equation*}
where one sums over an orthonormal frame and $F_L$ is the
curvature of the connection $A_L$ on the line bundle $L$. We denote
\begin{equation*}
\kappa^L_{g}:=\kappa_{g} + 4\mathcal{R}_L,
\end{equation*}
and we say that $\kappa^L_{g}$ is the \emph{$L$-twisted scalar
curvature}.  Notice that $\kappa^L_{g}$ depends on a choice of the
hermitian metric $h$ on $L$ and the connection $A_L$.

We need to consider coupling between the Riemannian
curvature and the curvature of the line bundle $L$ (which is just given
by an ordinary $2$-form $\omega$, which after dividing by $2\pi\,i$, has
integral de Rham class representing $c_1(L)$). Now recall
\cite[Lemma D.13]{lawson89:_spin}, which says that \emph{any}
$2$-form $\omega$ with $\frac{\omega}{2\pi\,i}$ in the de Rham class
of $c_1(L)$ can be realized as the curvature of some unitary
connection on $L$. We call such an $L$ a \emph{$\spinc$ line bundle}.
Now we define what we mean by $\spinc$ surgery.
\begin{definition}
Let $(M, L)$ be a closed $\spinc$ manifold (i.e., $M$ is a closed
oriented manifold and $L$ is a complex line bundle on $M$ with
$c_1(L)$ reducing mod $2$ to $w_2(M)$). We say that $(M', L')$ be
obtained from $(M,L)$ by \emph{$\spinc$ surgery in codimension $k$} if
there is a sphere $S^{n-k}$ embedded in $M$ with trivial normal
bundle, $M'$ is the result of gluing in $D^{n-k+1}\times S^{k-1}$ in
place of $S^{n-k}\times D^k$, and there is a $\spinc$ line bundle
$\cL$ on the trace of the surgery, a bordism
$(W, \cL)\co (M,L)\bord  (M'L')$,
such that $\cL$ restricts to $L$ on $M$ and to $-L'$ on
$M'$ respectively.
\label{def:spincsurg} 
\end{definition}
\begin{theorem}[Spin$^c$ surgery theorem, {\cite[Theorem 4.2]{BJ}}]
\label{thm:codim3spinc}
Let $(M, L)$ be a closed $n$-dimensional $\spinc$ manifold.  Assume
that $M$ admits a Riemannian metric $g$ and $L$ admits a hermitian
bundle metric $h$ and a unitary connection $A$ such that $\kappa^L_g
>0 $. Let $(M', L')$ be obtained from $(M,L)$ by $\spinc$ surgery in
codimension $k\ge 3$.  Then there is a metric $g'$ on $M'$, and $L'$
admits a hermitian bundle metric $h'$ and a unitary connection $A'$,
such that $\kappa^{L'}_{g'} > 0$.
\end{theorem}
This leads to the following $\spinc$ bordism theorem:
\begin{theorem}[Spin$^c$ bordism theorem, {\cite[Theorem 4.3]{BJ}}]
\label{thm:bordismspinc}
Let $(M,L)$ be a connected closed $n$-dimensional $\spinc$ manifold
which is not spin.  Assume that $M$ is simply
connected and that $n\ge 5$.  Also assume that there exists
a pair $(M',L')$ in the same bordism class in $\Omega^{\spinc}_n$
with a metric $g'$ on $M'$, a hermitian metric $h'$ and
a unitary connection $A'$ on $L'$ such that
$\kappa_{g'}^{L'} > 0$.  Then
$M$ admits a Riemannian metric
$g$ and $L$ admits a hermitian bundle metric $h$
and a unitary connection $A$ such that $\kappa_{g}^{L} > 0$.
\end{theorem} 
\begin{remark}
We emphasize that the condition that $(M,L)$ is $\spinc$, but not
$\spin$, is essential: one should not assume that if $(M,L)$ is
$\spinc$ and $\alpha^c(M,L)=0$, then one can choose a metric on $M$
and a hermitian metric and connection $A$ on $L$ so that
$\kappa_{g}^{L} > 0$, for this is false.  Indeed, suppose $M$ is
actually $\spin$ and $\dim M$ is $1$ or $2$ mod $8$ with
$\alpha(M)\neq 0$, so there is no psc-metric $g$ on $M$.  Adding in
the term $\mathcal R_L$ in this case only makes things worse, because
in suitable coordinates, $\mathcal R_L$ has the form
$\begin{pmatrix} \omega &0\\0&-\omega\end{pmatrix}$, where the
operator $\omega$ is constructed from the curvature of $L$, which can
be any exact $2$-form on $M$, so $\kappa_{g}^L $ cannot be strictly
positive in this case unless the scalar curvature $\kappa_{g}$ is
strictly positive, which is impossible.
\end{remark}
Even though vanishing of the index $\alpha^c(M,L)$ does not guarantee
that $\kappa_{g}^L >0$ for given $\spinc$-manifold $(M,L)$, we prove
that there is some representative in the same bordism class which has
psc-metric $g$ with $\kappa_{g}^L>0 $ for an appropriate
choice of bundle data.
\begin{theorem}[{\cite[Corollary 5.2]{BJ}}]
\label{cor:nonspinpsc1}
Let $(M,L)$ be a simply connected $\spinc$ manifold with
$\alpha^{\spinc}(M,L)\allowbreak
=0$ in $KU_n$. Then after changing $(M,L)$ up to
$\spinc$ cobordism, we can assume that $M$ admits a Riemannian
psc-metric $g$ and the line bundle $L$ over $M$ defining the $\spinc$
structure admits a hermitian metric $h$ and a connection $A$ such that
$\kappa_{g}^L >0$.
%
%
%
%
%
\end{theorem}
We notice that we do not have a dimensional restriction here; this is
because ``changing $(M,L)$ up to $\spinc$ cobordism'' makes the
problem of finding an appropriate psc metric and bundle data very
flexible. On the other hand, the following more elegant result holds
for non-spin $\spinc$ manifolds:
\begin{theorem}
\label{cor:nonspinpsc2}
Let $(M,L)$ be a simply connected non spin $\spinc$ manifold with
$\alpha^{\spinc}(M,L)=0$ in $KU_n$ with $n=\dim M\geq 5$.  Then $M$
admits a Riemannian psc-metric $g$, a hermitian metric $h$ and a
connection $A_L$ such that $\kappa_{g}^L >0$.
%
%
%
%
%
\end{theorem}

\noindent
Proofs of Theorems \ref{cor:nonspinpsc1} and \ref{cor:nonspinpsc2} are
based on the above bordism Theorem and studying the kernel of the index
homomorphism $\alpha^{\spinc}: \Omega^{\spinc}_n \to KU_n$.
To explain the idea, we first recall some basic facts about
$\spinc$ bordism, see  \cite[Chapter XI]{MR0248858}, 
\cite[\S8]{MR0234475}, and \cite{MR1166518}. In particular, we have an
isomorphism of bordism groups
\begin{equation*}
\Omega^{\spinc}_n\cong
\widetilde\Omega^{\spin}_{n+2}(\bC\bP^\infty),
\end{equation*}
where $\widetilde\Omega^{\spin}_{n+2}(\bC\bP^\infty)$ is the reduced
bordism group. Next, classes in $\spinc$ bordism are detected by their
Stiefel Whitney numbers (which are constrained just by the Wu
relations and the vanishing of $w_1$ and $w_3$) and integral
cohomology characteristic numbers (where in addition to the Pontryagin
classes, one can use powers of $c_1$ of the line bundle defining the
$\spinc$ structure) \cite[Theorem, p.\ 337]{MR0248858}. We do not need
to state all these results, however, we need a few examples.

We notice that the bordism class can change, depending on the choice
of $\spinc$ structure. Thus, for example,
$\Omega^{\spinc}_2\cong \bZ$, with all classes represented by
$(\bC\bP^1, L)$, $L$ a complex line bundle with $c_1(L)$ even, and the
isomorphism to $\bZ$ is given by
\begin{equation*}
\begin{array}{c}
(\bC\bP^1, L)\mapsto \frac12 \langle
c_1(L),\,[\bC\bP^1]\rangle.
\end{array}
\end{equation*}
Similarly, $\Omega^{\spinc}_4\cong \bZ^2$, with one generator given by
$(\bC\bP^1, \cO(2))^2$, with $\alpha^{\spinc}(\bC\bP^1, \cO(2))^2=1$,
and the other generator given by $(\bC\bP^2, \cO(1))$, where $c_1$ of
the anticanonical bundle $\cO(1)$ is the standard generator $x$ of
$H^2(\bC\bP^2;\bZ)$, on which $\alpha^{\spinc}$ takes the value $0$.
The calculation of $\alpha^{\spinc}$
on this generator is worked out by Hattori \cite{MR508087}
\begin{equation*}
\begin{aligned}
\alpha^c(\bC\bP^2,\cO(1))&=\ind\Dirac_{\bC\bP^2,\cO(1)} \\
&= \langle \widehat {\mathcal A}(\bC\bP^2)e^{x/2}, [\bC\bP^2]\rangle\\
&= \big\langle (1-\tfrac{1}{8}x^2)
(1 + \tfrac{1}{2}x + \tfrac{1}{2}\tfrac{x^2}{4}),
[\bC\bP^2]\big\rangle = 0,
\end{aligned}
\end{equation*}
by the Atiyah-Singer Theorem \cite[Theorem D.15, p.\
399]{lawson89:_spin}.

This last example turns out to be crucial, because there is a sense in
which $\bC\bP^2$ with the bundle $\cO(1)$, the dual of the
tautological bundle, generates the kernel of $\alpha^c$.  In more
detail, we use $(\bC\bP^2,\cO(1))$ to construct a transfer map
\begin{equation*}
T^{\spinc}\co \Omega^{\spinc}_n(BG)\to \Omega^{\spinc}_{n+4},
\end{equation*}
where $G$ is the Lie group $\SU(3)$, as follows.  The group $\SU(3)$
acts transitively on $\bC\bP^2\cong G/H$, where
$H=S(\U(2)\times \U(1))$, preserving the class of the bundle $\cO(1)$.
In particular, we obtain a fiber bundle $p: BH \to BG$ with
a fiber $\bC\bP^2$ and the structure group $\SU(3)$. 

In fact, the bundle $p: BH \to BG$ is a universal geometrical
$\bC\bP^2$-bundle for all $\bC\bP^2$-bundles with the structure group
$\SU(3)$. Thus given a $\spinc$ manifold $(M, L)$ and a map $f\co M\to
BG$, we can form the associated $\bC\bP^2$ bundle $\hat p\co E\to M$ as
a pull-back:
\begin{equation*}
\begin{diagram}
\setlength{\dgARROWLENGTH}{1.7em}
\node{E}
\arrow{e,l}{\hat f}
\arrow{s,l}{\hat p}
\node{BH} 
\arrow{s,l}{p}
\\
\node{M}
\arrow{e,l}{f}
\node{BG} 
\end{diagram}
\end{equation*}
where $E=M\times_f \bC\bP^2$ has dimension $n+4$ and has a $\spinc$
structure inherited from the $\spinc$ structure on $M$ defined by $L$
and the $\spinc$ structure on $\bC\bP^2$ defined by the bundle $\cO(1)$.

There is another transfer map introduced by Stolz in \cite{MR1189863}
and \cite{MR1259520}. This is defined similarly, but with
$G=\SU(3)$ replaced by $\PSp(3)$, $H=S(\U(2)\times \U(1))$ replaced
by $P(\Sp(2)\times \Sp(1))$, and $\bC\bP^2$ replaced by $\bH\bP^2$.
One obtains a transfer map
\begin{equation*}
T^{\spin}\co \Omega^{\spin}_n(B\!\PSp(3))\to \Omega^{\spin}_{n+8}.
\end{equation*}
Here is a main technical result we need:
\begin{theorem}
\label{thm:transfer}  
The above transfer maps
\begin{equation*}
T^{\spinc}\co \Omega^{\spinc}_{n-4}(B\SU(3))\to \Omega^{\spinc}_{n}\ \ \ \mbox{and}
\ \ \ 
T^{\spin}\co \Omega^{\spinc}_{n-8}(B\!\PSp(3))\to \Omega^{\spinc}_{n}
\end{equation*}
are such that
$\left\langle\mathrm{Im}(T^{\spinc})\cup \mathrm{Im}(T^{\spin})\right\rangle
= \mathrm{Ker }\ \alpha^{\spinc}$ as abelian groups.
\end{theorem}
\noindent
The proof of Theorem \ref{thm:transfer} requires some computations with
$\spinc$-bordism groups, see \cite[section 5]{BR}. Then Theorem
\ref{thm:transfer} implies that a $\spinc$-manifold $(M,L)$
$\spinc$-bordant to a union of total spaces $E$ of geometric
$\bC\bP^2$- and $\bH\bP^2$-bundles $E\to B$.

These cases are slightly different. Consider geometric $\bC\bP^2$-bundles
first.  We start with a trivial bundle, the $\spinc$-manifold
$(\bC\bP^2,\cO(1))$.  Here we use the Fubini-Study metric $g_{FS}$
along with the usual connection on the dual of the tautological
bundle.  Then if $\omega$ is the K\"ahler form, this is also the
curvature of the bundle $\cO(1)$ and the Ricci tensor is 6 times the
metric.  It is well-known that $\frac14\kappa_{FS} =6$, and the minimal
eigenvalue of $\cR$ is $-2$, so
$\frac14\kappa_{g_{FS}}^{\cO(1)}=\frac14\kappa_{g_{FS}}+\cR \geq 6-2>0$.

Recall that $\SU(3)$ acts by isometries of the standard Fubini-Study
metric. Then a total space $E$ of a geometric $\bC\bP^2$-bundles $E\to
B$ has a canonical line bundle $L\to E$ which restricts to the bundle
$\cO(1)$ on the fibers. Then by choosing the metric and connection so
that on each fiber, we have a very small multiple of the Fubini-Study
metric $g_{FS}$ and the curvature of the line bundle is the K\"ahler
form, the curvature of the fibers will dominate everything else.

Let $E\to B$ be a geometric $\bH\bP^2$-bundle over a $\spinc$-manifold
$B$, where the structure group $\PSp(3)$ is an isometry group of the
standard metric $g_{\bH\bP^2}$ (of positive curvature).  Then we can
rescale the metric $g_{\bH\bP^2}$ so that its scalar curvature will
dominate everything else. These arguments prove Theorem
\ref{cor:nonspinpsc1}. To prove Theorem \ref{cor:nonspinpsc2} we have
to analyze the image of $\Omega^{\spin}_n$ in the group
$\Omega^{\spinc}_n$.
\subsection{Finding a well-adapted psc-metric on a simply-connected
  $\spinc$-manifold}
Let $(M,L)$ be a $\spinc$-manifold, where $M=X\cup_{\p X=Y} -N(B)$ and
$B$ is dual to $L$. We have identified two primary obstructions
$$ \alpha^{\spinc}(M,L)\in KU_n \ \ \mbox{and} \ \ \ \alpha(B)\in
KO_{n-2},$$ and a secondary obstruction $\alpha^{\mathrm{rel}}(X,Y,
g_Y)\in KO_n$ for existence of a well-adapted psc-metric on
$\spinc$-manifold $(M,L)$.  The secondary obstruction
$\alpha^{\mathrm{rel}}(X,Y, g_Y)$ depends on a choice of a psc-metric
$g_Y$ on $Y$ which is determined by a choice of a psc-metric $g_B$ on
$B$. Then when we say that $\alpha^{\mathrm{rel}}(X,Y, g_Y)$ vanishes,
we mean that \emph{there exists a psc-metric $g_B$ which determines
  the metric $g_Y$ so that $\alpha^{\mathrm{rel}}(X,Y, g_Y)=0$ in
  $KO_n$}.

We emphasize that if $B$ is simply-connected with $\dim B \geq 5$, then if 
$\alpha(B)\in KO_{n-2}$ vanishes, then there exists a psc-metric $g_B$ on $B$.
\begin{theorem}\label{spinc-well-adap}
Let $(M,L)$ be a non-spin $\spinc$-manifold, where
$M=X\cup_{\p X=Y} -N(B)$ and $B$ is dual to the line bundle $L$, where
$M$ and $B$ are simply-connected and $\dim M=n \geq 7$.  Assume that the
primary obstructions $\alpha^{\spinc}(M,L)\in KU_n$ and $\alpha(B)\in
KO_{n-2}$ vanish and the secondary obstruction
$\alpha^{\mathrm{rel}}(X,Y, g_Y)=0$ for some choice of a metric $g_B$
(and, consequently of $g_Y$). Then $(M,L)$ admits a well-adapted
psc-metric.
\end{theorem}  
\begin{proof}
Since $\alpha^{\spinc}(M,L)$ vanishes in $KU_n$, Theorem
\ref{cor:nonspinpsc2} says that $M$ admits a psc-metric $g$ so that
$\kappa^{L}_{g}$ is positive definite.  This by itself is not good enough,
since this metric may not be well-adapted with respect to $B$.

However an analysis of relevant bordism groups shows that
$(M,L)$ is $\spinc$-bordant to a disjoint union $M''\sqcup (M',L')$
in the following sense:  the manifold $M''$ is a closed $\spin$ manifold,
and $(M',L')$ is a $\spinc$ pair, such that $L'$ is trivial
away from another closed spin manifold $B'$, and
\begin{equation*}
(B,L|_B)\sim (B',L'|_{B'}) \ \ \mbox{in} \ \ \ \Omega_{n-2}^\spin(\bC\bP^\infty)
\end{equation*}
while $M\sim M''\sqcup M'$ in $\Omega_{n}^{\spinc}$. We can also take
$M'$ and $M''$ to be simply connected. Then, since the $\alpha$
invariants only depend on $\spin/\spinc$ bordism classes and are linear
on the bordism groups, $\alpha^{\spinc}(M',L')=0$ and $\alpha(M'')=0$.
Then we can construct $M'$ so that it has a well-adapted psc-metric,
by a slight refinement of Theorem \ref{cor:nonspinpsc2}.  Also, $M''$
has a psc-metric by Stolz's Theorem.  Putting everything together, we
can push the well-adapted psc-metric on $M''\sqcup M'$ through the
bordism to get a well-adapted psc-metric on $M$, using the
Gromov-Lawson surgery technique.  First, we have to do 
codimension $3$ surgeries on $B'$ to convert it to $B$, and use these
surgeries to push the metric on the tubular neighborhood.  Then do
surgeries on the interior to push the psc-metric from $X'$ to $X$.
\end{proof}

\section{More General Manifolds with Singularities}
\label{sec:sing}
The situations discussed in sections \ref{sec:pin} and \ref{sec:spinc},
as well as the paper \cite{MR1857524}, lead to a more general subject,
the classification of ``manifolds with singularities'' or
\emph{singular spaces}
admitting a psc-metric.  Here by ``manifolds with singularities'' we mean
compact Hausdorff spaces with a stratification, where the strata
are locally closed subspaces (hence locally compact) which are themselves
smooth manifolds, and usually one adds a condition on the local structure of
a neighborhood of each stratum in a larger one.  Various categories
of manifolds with singularities are discussed in some detail in the book
\cite{MR1308714}.  The prototypes of singular spaces, which are good
examples to keep in mind, are either projective varieties over $\bR$
or $\bC$ which are not necessarily smooth, or else quotient spaces
$M/G$ of a smooth manifold by a smooth action of a Lie group $G$,
in the case where there can be more than one orbit type.
Here we will simplify the discussion by restricting
attention to the case of only two strata.  Thus if $M$ is such a singular
space, $M$ has a dense open subset $\mathring M$ which is a smooth
manifold, and $M\smallsetminus \mathring M = \beta M$ is a closed
manifold of smaller dimension (possibly disconnected).
The reasons for the local neighborhood condition are:
\begin{enumerate}
\item Such conditions hold in the two kinds of prototypes: algebraic
  varieties and quotients of smooth actions.
\item Without such a condition one can have very wild examples.  For example,
  take a smooth manifold $M$ and collapse some closed subset $X$ to a
  point.  The quotient space $M/X$ is the union of the open manifold
  $\mathring M = M\smallsetminus X$ and a point, but if $X$ is pathological,
  the neighborhood of the singular point can be very complicated.
\end{enumerate}  

In the rest of this section we will consider singular manifolds
(more precisely, ``pseudomanifolds'') with exactly two strata,
$\mathring M$ and $M\smallsetminus \mathring M = \beta M$, such that
$\beta M$ has a tubular neighborhood in $M$ homeomorphic to a
fiber bundle $c(L) \to N(\beta M)\to \beta M$, where the
fibers are the cone $c(L)= L\times[0,1]/(L\times \{0\})$
on a fixed closed Riemannian manifold $(L,g_L)$, called
the ``link,'' and the structure group of the bundle is contained in
the isometry group of $(L,g_L)$ (extended to act on $c(L)$
so as to preserve the distance to the cone point).
The fiber bundle has a natural
section embedding $\beta M$ in $N(\beta M)$ as the union of the
``vertex points'' of the cone fibers.  Unless $L$ is a sphere,
which \emph{was} the case in sections \ref{sec:pin} and \ref{sec:spinc},
where we had $L=S^0$ in section \ref{sec:pin} and $L=S^1$ in section
\ref{sec:spinc}, such
a pseudomanifold is generally not even homeomorphic to a topological manifold
(let alone to a smooth manifold), so it certainly doesn't admit Riemannian
metrics in the usual sense.  So if $L$ is not a sphere,
what do we mean by a psc-metric on $M$?
Extrapolating from the cases we have discussed in
sections \ref{sec:pin} and \ref{sec:spinc}, we restrict attention to
``well-adapted'' metrics with respect to the neighborhood structure
near the singular stratum. That means we impose the following
conditions.
\begin{definition}
\label{def:adaptedmetric2}
Let $M^n$ be a stratified (compact) pseudomanifold with two strata
$\mathring M$ and $M\smallsetminus \mathring M = \beta M$ as above,
with $\beta M$ having a tubular neighborhood $N(\beta M)$ which is a
$c(L)$-bundle over $\beta M$, for a fixed link $L$. Recall that
$\mathring M$ is a smooth $n$-manifold and that $\beta M$ is a closed
manifold of smaller dimension.  A \emph{well-adapted
metric} on $M$ will mean the following:
\begin{enumerate}
\item A choice of a Riemannian metric $g$ on $\mathring M$ and of a Riemannian
metric $g_\beta$ on $\beta M$, such that:
\item the restriction of $g$ to
$X = M\smallsetminus \text{int}\,N(\beta M)$ is a product metric
$g_{\p X} + dt^2$ in a collar neighborhood of $\p X$,
\item the map $p\co (\p X,g_{\p X}) \to (\beta M, g_\beta)$
is a Riemannian submersion with the given metric on $L$ as
the ``vertical metric'' on the fibers, and
\item in a slightly smaller neighborhood of $\beta M$, with fibers
$L\times[0,1-\varepsilon]/(L\times \{0\})$, $g$ has the local form
$dr^2+ r^2g_L+p^*g_\beta$, where $r$ is the distance to the vertex of
the cone in $c(L)$.
\end{enumerate}
\end{definition}
One can easily see that Definition \ref{def:adaptedmetric2} specializes
to Definition \ref{def:adaptedmetric1} when $L=S^0$.
\begin{definition}
\label{def:adaptedpscmetric}
If $M^n$ is a stratified pseudomanifold with two strata as in Definition
\ref{def:adaptedmetric2}, we say that a \emph{well-adapted psc-metric}
on $M$ is a well-adapted metric on $M$ in the sense of Definition
\ref{def:adaptedmetric2}, such that $g$ and $g_\beta$ are psc-metrics
on $\mathring M$ and $\beta M$, respectively. Again this agrees with
our earlier terminology.
\end{definition}
The basic question we want to study is this:
\begin{question}
\label{q:adpatedpsc1}  
  Suppose $M^n$ is a (compact) pseudomanifold with two strata as above.
  Clearly, if $M$ admits a well-adapted psc-metric, then
  the closed manifold $\beta M$ admits a psc-metric.  What additional
  conditions are needed to ensure the converse, at least if $n$ is
  sufficiently large?
\end{question}

We have studied this question in \cite{BJ,BJP,BJP2}.  Basically, in order
to get any good results on this question, we want the link manifold $L$
to be a homogeneous space for a compact Lie group $G$, with the metric
on $L$ to be $G$-invariant.  Such homogeneous spaces always have
nonpositive curvature, and there are two rather different cases to
consider: the case where $G$ is a torus, in which case $L$ itself
is necessarily a torus and we might as well take $G=L$
with a flat invariant metric, or the
case where $G$ is compact semisimple, in which case $L$ is a manifold
of $G$-invariant positive sectional curvature.  The prototype of the
first case is the case where $G=L=S^1$, treated in section \ref{sec:spinc}
above and in more detail in \cite{BJ}.  In this case, since the cone on
a circle is a disk, our pseudomanifold $M$ is actually a smooth manifold,
and a well-adapted psc-metric on $M$ is in particular a psc-metric on $M$
in the usual sense for smooth manifolds.  If $M$ has a spin structure
then $\alpha(M)$ is an obstruction to such a metric, in addition to
whatever obstruction there might be to a psc-metric on $\beta M$.
However, if we decompose $M$ as $X\cup_{\p X} N(\beta M)$ as above,
and if $\p X$ and $\beta M$ are both simply connected, then from the
Gysin sequence and long exact homotopy sequence of the circle bundle
$S^1\to \p X\to \beta M$, one can see that $c_1$ of the circle bundle
(which one can identify with $c_1$ of a complex line bundle for which
$N(\beta)$ is the unit disk bundle) has to be non-zero mod $2$,
and thus $\beta M$ and $M$ cannot be spin simultaneously.  But if can happen,
for example, that $M$ is spin and $\beta M$ is spin$^c$.  An example
given in \cite[Remark 3.6]{BJ} is the case where $M=\bC\bP^5\# \Sigma^{10}$,
where $\Sigma^{10}$ is a homotopy sphere with $\alpha(\Sigma^{10})\ne 0$
in $KO_{10}\cong \bZ/2$, and $\beta M=\bC\bP^4$, which obviously admits
a psc-metric.  In this case, the $\alpha$-invariant shows that $M$
does not admit any psc-metric, let alone a well-adapted one.

A case not discussed in \cite{BJ} is the case where $L$ is a
higher-dimensional torus.  In this case, the ``singular manifold'' is
genuinely singular, since the cone on $L$ is not locally Euclidean.
In this case, \emph{well-adapted psc-metrics never exist}, because
of the following calculation:
\begin{lemma}[{\cite[Lemma 3.1]{BJP}}]
\label{lem:scalofcone}
The scalar curvature function on the cone $c(L)$ on $L$ with the
conical metric $dr^2+ r^2g_L$ {\lp}with the
vertex $r=0$ of the cone deleted{\rp} is $(\kappa_L-\kappa_\ell)r^{-2}$,
where $\kappa_L$ is the scalar curvature of $L$ and $\kappa_\ell$
is the scalar curvature of a
standard round sphere $S^\ell(1)$ of radius $1$, $\ell=\dim L$.
\end{lemma}
\begin{corollary}
\label{cor:higherdimtorus}
If $L$ is flat and $\dim L>1$, then the scalar curvature of $c(L)$
tends to $-\infty$ as $r\to 0$, and thus a well-adapted psc-metric
is impossible.
\end{corollary}
\begin{proof}
Just take $\kappa_L=0$ and $\kappa_\ell=\ell(\ell-1)$ in the above formula.
For the application to well-adapted metrics, observe that such a metric
is locally a Riemannian product of $c(L)$ and $\beta M$ up to
small corrections coming from the curvature of the fiber bundle,
and so there is no way to overcome the hugely negative scalar curvature
on the cone.
\end{proof}

The results in \cite{BJP,BJP2} have to do with the other case where
$L$ is a homogeneous space of a compact semisimple Lie group $G$, and the
bundle $p\co \p X\to \beta M$ is an associated bundle $P\times_G L$
to a principal $G$-bundle $G\to P\to \beta M$.  Because of Lemma 
\ref{lem:scalofcone}, we take the scalar curvature of $L$
to be equal to $\kappa_\ell=\ell(\ell-1)$ (if you like, this is a
normalization of the cone angle) so that the cone $c(L)$ is scalar-flat.
This ensures that if $\beta M$ has a psc-metric, we can lift this
metric to a well-adapted psc-metric on the tubular neighborhood
$N(\beta M)$.  This follows from an application of the O'Neill
formulas for the curvature of a Riemannian submersion.

\begin{proposition}[{\cite[Theorem 3.5]{BJP}}]
\label{prob:obstr}
Let $L = G/H$ be a homogeneous space, $\dim L = \ell$, where $G$
is a connected compact semisimple Lie group, and $g_L$ be a $G$-invariant
Riemannian metric on $L$ of constant scalar curvature equal to
$\kappa_\ell=\ell(\ell-1)$.  Let $M=X\cup_{\p X}N(\beta M)$ be a compact
pseudomanifold of dimension $n$ with two strata,
$\mathring M\cong \text{int}\,X$ and $\beta M$, where $\p X= P\times_G L$,
$P$ a principal $G$-bundle over $\beta M$.  If $M$ admits a well-adapted
metric of positive scalar curvature, then $\beta M$ admits a psc-metric.
\end{proposition}

When $X$ and $\beta M$ are both $\spin$ manifolds, this gives us two
obstructions to a well-adapted psc-metric, the $\alpha$-invariant
of $\beta M$ (or its ``higher'' generalization if $\beta M$ is not
simply connected), and the secondary invariant
$\alpha^{\mathrm{rel}}(X,\p X, g_{\p X})$, which depends on the choice of a
psc-metric on $\beta M$ (or equivalently, on the Riemannian submersion
metric on $\p X$).

In the papers \cite{BJP,BJP2}, we were able to show that the vanishing
of these obstructions is sometimes sufficient for $M$ to admit a well-adapted
psc-metric.  Not only that, but in some cases we are able to obtain
information on the topology of space of well-adapted psc-metrics on $M$.
For details, we refer the reader to those other papers, but here we just
give an indication of some of the key techniques.

The main method for proving existence of well-adapted psc-metrics
is to introduce a suitable notion of (spin) bordism for pseudomanifolds
in the appropriate category, giving rise to an exact sequence of
bordism groups of the form
\begin{equation}
\label{eq:bordismseq}
\cdots \to \Omega^{\spin}_{n} \xrightarrow{i}  
\Omega_n^{\spin,  (L,G)\fb}
\xrightarrow{\beta} \Omega^{\spin}_{n-\ell-1}(BG)
\xrightarrow{t} \Omega^{\spin}_{n-1} \to \cdots.
\end{equation}
Here $\Omega_n^{\spin,  (L,G)\fb}$ is the bordism group of $n$-dimensional
spin pseudomanifolds with $(L,G)$-fibered singularities, i.e., with
the structure we have been talking about, where the fibration
$\p X\to \beta M$ comes from a principal $G$-bundle over $\beta M$.
The ``inclusion'' map $i$ comes from viewing a closed spin manifold
as such a pseudomanifold with empty singularities, and the ``Bockstein''
map $\beta$ sends $(M,X,\p X\to \beta M)$ to the bordism class of $\beta M$
with its map to $BG$ classifying the principal $G$-bundle that defines
the $(L,G)$-fibered singularity structure.  The ``transfer map'' $t$
sends the bordism class of $\beta M\to BG$ to the total space of the
associated $L$-bundle.  The exact sequence \eqref{eq:bordismseq},
along with its generalization to the case where the manifolds
are not simply connected, along with the surgery method of Gromov
and Lawson, is the main technical tool for proving positive results.

Here is the statement of the main theorem of \cite{BJP}:
\begin{theorem}[{Existence Theorem \cite[Theorem 1.2]{BJP}}]
\label{thm:pscsufficiency} 
Let $M=X\cup_{\p X}N(\beta M)$ be an $n$-dimensional compact
pseudomanifold with $X$ and $\beta M$ spin and simply connected.
Assume that the fiber bundle $\phi\co \p M\to \beta M$ has fiber
$L=G/H$ and is the associated bundle to a principal $G$-bundle
over $\beta M$, where $G$ is a simply connected compact Lie group.
Furthermore,
assume $n\geq \ell+6$, where $\ell=\dim L$, and that one of the
following condition holds:
\begin{enumerate}
\item[{\rm (i)}] either $L$ is a spin $G$-boundary of a manifold
    $\bar L$ equipped with a psc-metric $g_{\bar L}$,
    which is a product near the boundary and satisfies $g_{\bar L}|_L=g_L$;
\item[{\rm (ii)}] or $\p M=\beta M\times L$,
    where $L$ is an even quaternionic projective space.
\end{enumerate}
Then the vanishing of the primary and secondary obstruction invariants
$\alpha(\beta M) \in KO_{n-\ell-1}$ and
$\alpha^{\mathrm{rel}}(X,\p X, g_{\p X})\in KO_n$
implies that $M$ admits an adapted psc-metric.
\end{theorem}

Case (i) of this theorem holds if $L=G$ or if $L$ is a sphere or an
even complex projective space.  In these cases, the idea of the proof
is to show that $M$ can be obtained via surgery from the
disjoint union of a psc spin manifold and the result of replacing
$X$ by a $\bar L$-bundle over $\beta M$, which clearly has a well-adapted
psc-metric.  Case (ii) uses a different idea; in  this case, since
even quaternionic projective spaces are not zero-divisors in
$\Omega^{\spin}_*$, the exactness of \eqref{eq:bordismseq} forces
$[\beta M\to BG]$ to vanish in $\Omega^{\spin}_{n-\ell-1}(BG)$, and this
can be used with the surgery technique to show that $M$ admits
a well-adapted psc-metric.

The sequel paper \cite{BJP2} contains a generalization of Theorem
\ref{thm:pscsufficiency}  to the non-simply connected case, as well
as some results of the topology of the space of well-adapted psc-metrics
when this space is non-empty.  One interesting result along these lines
is the following, which is nontrivial even in the simply connected case.
\begin{theorem}[{\cite[Theorem 1.5]{BJP2}}]
\label{thm:inj}
Let $M=X\cup_{\p X}N(\beta M)$ be an $n$-dimensional compact
pseudomanifold. Assume that the fiber bundle $\phi\co \p M\to \beta M$ has
fiber $L=G/H$ and is the associated bundle to a principal $G$-bundle
over $\beta M$, and assume $M$ admits a well-adapted psc-metric.
Then the homotopy groups of the space of psc-metrics on $\beta M$
inject into the homotopy groups of the space of
well-adapted psc-metrics on $M$.
\end{theorem}

\section*{Acknowledgements}
JR was partially supported by {U.S.} NSF grant number DMS-1607162.\\
BB was partially supported by Simons collaboration grant 708183.
\bibliographystyle{amsplain}
\bibliography{PSCFiberedSing}

\end{document}